\documentclass[12pt]{article}
\usepackage{amsmath,amssymb}
\usepackage{amsthm}

\def\D{\mathbf{D}}

\def\E{\mathbf{E}}

\def\N{\mathbf{N}}

\def\P{\mathbf{P}}
\def\R{\mathbf{R}}

\def\1{\mathbf{1}}

\def\al{\alpha}
\def\be{\beta}
\def\pa{\partial}
\def\ep{\epsilon}
\def\de{\delta}
\def\ga{\gamma}

\newtheorem{prop}{Proposition}[section]
\newtheorem{theorem}{Theorem}[section]

\newtheorem{remark}{Remark}

\newcommand{\la}{\lambda}

\newcommand{\om}{\omega}
\newcommand{\Ga}{\Gamma}

\newcommand{\Om}{\Omega}

\begin{document}
\title{On fully mixed and multidimensional extensions of the Caputo and Riemann-Liouville derivatives,
 related Markov processes and fractional differential equations
}
\author{
Vassili N. Kolokoltsov\thanks{Department of Statistics, University of Warwick,
 Coventry CV4 7AL UK, and associate member  of IPI RAN RF, Email: v.kolokoltsov@warwick.ac.uk}}
\maketitle

\begin{abstract}
From the point of view of stochastic analysis the Caputo and Riemann-Liouville derivatives
of order $\al \in (0,2)$ can be viewed as (regularized) generators of stable L\'evy motions interrupted
on crossing a boundary. This interpretation
naturally suggests fully mixed, two-sided or even multidimensional generalizations
of these derivatives, as well as a probabilistic approach to the analysis of the related equations.
These extensions are introduced and some well-posedness results are obtained that generalize, simplify
 and unify lots of known facts. This probabilistic analysis leads one to study a
 class of Markov processes that can be constructed from any given Markov process in $\R^d$ by blocking
 (or interrupting) the jumps that attempt to cross certain closed set of 'check-points'.
\end{abstract}


\section{Introduction}

For a general background in fractional calculus and fractional equations we refer to books
\cite{Dietbook}, \cite{KiSrTr06}, \cite{Podlub99},
\cite{PskhuRusBook}, see also survey \cite{MaiGo07}, 
for the crucial link with CTRW to \cite{Scal12}, \cite{Ko09} and \cite{MeerStr14},
and for the numerous applications in natural science to
\cite{TarBook11} and \cite{UchBook}.

The aim of this paper is to present a systematic treatment of a class of equations that include
fractional derivatives as very particular cases. Unlike the mostly analytic studies
of fractional differential equations (see the reference above), the present treatment and the corresponding far reaching
extensions of fractional derivatives are based on a probabilistic point of view.
This link with probability provides a powerful tool for the study of fractional equations.
It is also worth mentioning the recent activity on
proving probabilistic interpretation of solutions by analytic methods, see e.g. \cite{GorLuchSto},
while our approach provides such an interpretation as a starting point.

From this point of view the basic Caputo and Riemann-Liouville (RL) derivatives
of order $\al \in (0,2)$ can be viewed as (regularized) generators of stable L\'evy motions interrupted
on crossing a boundary. This interpretation
naturally suggests fully mixed, two-sided or even multidimensional generalizations
of these derivatives, as well as a probabilistic approach to the analysis of the related equations.
These extensions are introduced leading to well-posedness results that generalize, simplify
and unify lots of known facts. Some explicit solutions are also obtained.
The corresponding probabilistic analysis leads one to study an interesting general
class of Markov processes that can be constructed from any given Markov process in $\R^d$ by blocking
(or interrupting) the jumps that attempt to cross certain closed set of check-points.
This analysis is only initiated in the present work.
Further development, as well as the application of this technique to the study of
fractional in time and space diffusion equations and 
to the convergence of CTRW, will be discussed in
separate publications.

The paper is organized as follows. In the next preliminary section we introduce in a convenient form
the main objects of fractional calculus, the Caputo and RL derivatives.

In Section  \ref{secprointfrac} we explain in detail the probabilistic meaning of these derivatives
and their natural place in stochastic analysis leading, on the one hand side, to far reaching generalizations,
and on the other hand, to a unified treatment of various equations by powerful tools of stochastic analysis.
We distinguish the cases of $\be \in (0,1)$ and $\be \in (1,2)$, because in the first case
the Caputo derivative has a direct probabilistic interpretation and in the second an additional
regularization is needed. We also treat separately a multi-dimensional extension, as it includes one
additional ingredient, the projection of a jump on a boundary.

The final Section \ref{secwelposedfrac} initiates the rigorous theory of the equations
and processes introduced above by providing some basic examples (not at all exhaustive)
of well-posedness results and explicit solutions that
can be obtained by these tools, with main attention restricted to the analogs
of the derivatives of order $\be \le 1$.

Appendix derives some equivalent versions of basic fractional derivatives mentioned in
the next section without proof. These calculations should be obvious for specialists in fractional calculus and
are given here for completeness.

We shall denote by $\1_M$ the indicator function of a set $M$.

\section{Preliminaries: classical fractional derivatives}

For convenience we recall here the basic definitions of fractional derivatives and their equivalent
representations fitting our purposes.

We shall repeatedly use the Euler identity $\Ga (x)=(x-1)\Ga (x-1)$ for the
Euler Gamma-function whenever appropriate without mentioning it.

Due to the formula for the iterated Riemann integral
\begin{equation}
\label{eqdefRLiteratedint}
I_{a}^{n} f(x)=\frac{1}{(n-1)!}\int_a^x (x-t)^{n -1}f(t) dt,
\end{equation}
it is natural to extend it analytically, if $x>a$, to complex $n$ with positive real part,
leading to the following definition of the (right) fractional or Riemann-Liouville (RL)
integral of order $\be$ (with positive real part):
\begin{equation}
\label{eqdefRLint}
I_a^{\be}f(x)=I_{a+}^{\be} f(x)=\frac{1}{\Ga (\be)}\int_a^x (x-t)^{\be -1}f(t) dt.
\end{equation}
Notice that for $x<a$, we have negative numbers $(x-t)$ in power $n-1$, so that the corresponding extension
to complex (or even real) $n$ leading to the so-called left $RL$ integrals have some subtleties to be discussed later.

Noting that the derivation is the inverse operation to usual integration, the above definition
 suggests two notions of fractional derivative, the so-called
RL (right) derivatives of order $\be >0, \be \notin \N$:
\begin{equation}
\label{eqdefRLder}
D^{\be}_{a+}f(x)=\frac{d^n}{dx^n} I_{a+}^{n-\be}f(x)=\frac{1}{\Ga (n-\be)}\frac{d^n}{dx^n} \int_a^x (x-t)^{n-\be -1}f(t) dt, \quad x>a,
\end{equation}
and the so-called Caputo (right) derivative of order $\be >0, \be \notin \N$:
\begin{equation}
\label{eqdefCder}
D^{\be}_{a+\star}f(x)=I_a^{n-\be} \left[\frac{d^n}{dx^n}f\right](x)
=\frac{1}{\Ga (n-\be)}\int_a^x (x-t)^{n-\be -1}\left[\frac{d^n}{dt^n} f\right](t) dt, \quad x>a
\end{equation}
where $n$ is the maximal integer that is strictly less than $\be+1$.

Straightforward integration by parts (see Appendix) show that, for smooth enough $f$ and $\be \in (0,1)$, $x>a$,
\begin{equation}
\label{eqRLder0}
D^{\be}_{a+}f(x)
=\frac{1}{\Ga (-\be)} \int_0^{x-a}\frac{f(x-z)-f(x)}{z^{1+\be}}dz
+\frac{f(x)}{\Ga (1-\be) (x-a)^{\be}},
\end{equation}
\begin{equation}
\label{eqCder0}
D^{\be}_{a+\star}f(x)
=\frac{1}{\Ga (-\be)} \int_0^{x-a}\frac{f(x-z)-f(x)}{z^{1+\be}}dz
+\frac{f(x)-f(a)}{\Ga (1-\be) (x-a)^{\be}},
\end{equation}
implying
\begin{equation}
\label{eqRLCapder}
D^{\be}_{a+ \star}f(x)=D^{\be}_{a+}[f-f(a)](x)
=D^{\be}_{a+ }f(x)-\frac{f(a)}{\Ga (1-\be)|x-a|^{\be}}.
\end{equation}

In particular it follows that for smooth bounded integrable functions, the right RL and Caputo derivatives
 coincide for $a=-\infty$, $\be \in (0,1)$,
and one defines the fractional derivative in generator form as their common value:
\begin{equation}
\label{eqRLder0inf}
\frac{d^{\be}}{dx^{\be}}f(x)=D^{\be}_{-\infty+}f(x)=D^{\be}_{-\infty+\star}f(x)
=\frac{1}{\Ga (-\be)} \int_0^{\infty}\frac{f(x-z)-f(x)}{z^{1+\be}}dz.
\end{equation}

Analogously (see Appendix for detail), for smooth enough $f$ and $\be \in (1,2)$, $x>a$, one finds that
\begin{equation}
\label{eqRLder1}
D^{\be}_{a+}f(x)
=\frac{1}{\Ga (-\be)} \int_0^{x-a}\frac{f(x-z)-f(x)+f'(x)z}{z^{1+\be}}dz
+\frac{f(x)(x-a)^{-\be}}{\Ga (1-\be)}+\frac{\be f'(x)(x-a)^{1-\be}}{\Ga (2-\be)},
\end{equation}
\[
D^{\be}_{a+\star}f(x)
=\frac{1}{\Ga (-\be)} \int_0^{x-a}\frac{f(x-z)-f(x)+f'(x)z}{z^{1+\be}}dz
\]
\begin{equation}
\label{eqCder1}
+\frac{(f(x)-f(a))(x-a)^{-\be}}{\Ga (1-\be)}+\frac{(\be f'(x)-f'(a))(x-a)^{1-\be}}{\Ga (2-\be)},
\end{equation}
so that
\begin{equation}
\label{eqRLCder2}
D^{\be}_{a+\star}f(x)=D^{\be}_{a+}[f-f(a)-f'(a)(.-a)](x)
=D^{\be}_{a+}f(x)-\frac{f(a)(x-a)^{-\be}}{\Ga (1-\be)}-\frac{f'(a)(x-a)^{1-\be}}{\Ga (2-\be)}.
\end{equation}

Again for smooth bounded integrable functions, the right RL and Caputo derivatives
 coincide for $a=-\infty$, $\be \in (1,2)$,
and one defines the fractional derivative in generator form as their common value:
\begin{equation}
\label{eqRLder2}
\frac{d^{\be}}{dx^{\be}}f(x)=D^{\be}_{-\infty+}f(x)=D^{\be}_{-\infty+\star}f(x)
=\frac{1}{\Ga (-\be)} \int_0^{\infty}\frac{f(x-z)-f(x)+f'(x)z}{z^{1+\be}}dz.
\end{equation}

Turning to the left derivative notice that for $x<a$ formula \eqref{eqdefRLiteratedint}
rewrites as
\begin{equation}
\label{eqdefRLiteratedintleft}
I_{a}^{n} f(x)=\frac{(-1)^n}{(n-1)!}\int_x^a (t-x)^{n -1}f(t) dt,
\end{equation}
suggesting several possible normalizations for the analytic continuation in $n$
and the corresponding inversions (fractional derivatives).
We are interested only in the derivatives of order less than $2$.
For a unified probabilistic interpretation of these derivatives it is convenient to choose the left versions
of \eqref{eqdefRLder}, \eqref{eqdefCder} as follows.
For $\be \in (0,1)$:
\begin{equation}
\label{eqdefRLderleft1}
D^{\be}_{a-}f(x)=-\frac{1}{\Ga (1-\be)}\frac{d}{dx} \int_x^a (t-x)^{-\be}f(t) \, dt, \quad x<a,
\end{equation}
\begin{equation}
\label{eqdefCderleft1}
D^{\be}_{a-\star}f(x)=-\frac{1}{\Ga (1-\be)} \int_x^a (t-x)^{-\be}f'(t) \, dt, \quad x<a;
\end{equation}
for $\be \in (1,2)$:
\begin{equation}
\label{eqdefRLderleft2}
D^{\be}_{a-}f(x)=\frac{1}{\Ga (2-\be)}\frac{d^2}{dx^2} \int_x^a (t-x)^{1-\be}f(t) \, dt, \quad x<a,
\end{equation}
\begin{equation}
\label{eqdefCderleft2}
D^{\be}_{a-\star}f(x)=\frac{1}{\Ga (2-\be)}\int_x^a (t-x)^{1-\be}f''(t) \, dt, \quad x<a.
\end{equation}

When $\be \in (0,1)$ and $x<a$, similar calculations as for the right derivative (see \eqref{eqcalRLder})
 lead to the following analogs of
\eqref{eqRLder0}, \eqref{eqCder0}:
\begin{equation}
\label{eqRLder0left}
D^{\be}_{a-}f(x)
=\frac{1}{\Ga (-\be)} \int_0^{a-x}\frac{f(x+z)-f(x)}{z^{1+\be}}dz
+\frac{f(x)}{\Ga (1-\be) (a-x)^{\be}},
\end{equation}
\begin{equation}
\label{eqCder0left}
D^{\be}_{a-\star}f(x)
=\frac{1}{\Ga (-\be)} \int_0^{a-x}\frac{f(x+z)-f(x)}{z^{1+\be}}dz
+\frac{f(x)-f(a)}{\Ga (1-\be) (a-x)^{\be}},
\end{equation}
implying
\begin{equation}
\label{eqRLCapderleft}
D^{\be}_{a- \star}f(x)=D^{\be}_{a-}[f-f(a)](x)
=D^{\be}_{a- }f(x)-\frac{f(a)}{\Ga (1-\be)(a-x)^{\be}}.
\end{equation}

When $\be \in (1,2)$, $x<a$, one obtains
\begin{equation}
\label{eqRLder1left}
D^{\be}_{a-}f(x)
=\frac{1}{\Ga (-\be)} \int_0^{a-x}\frac{f(x+z)-f(x)-f'(x)z}{z^{1+\be}}dz
+\frac{f(x)(a-x)^{-\be}}{\Ga (1-\be)}-\frac{\be f'(x)(a-x)^{1-\be}}{\Ga (2-\be)},
\end{equation}
\[
D^{\be}_{a-\star}f(x)
=\frac{1}{\Ga (-\be)} \int_0^{a-x}\frac{f(x+z)-f(x)-f'(x)z}{z^{1+\be}}dz
\]
\begin{equation}
\label{eqCder1left}
+\frac{(f(x)-f(a))(a-x)^{-\be}}{\Ga (1-\be)}-\frac{(\be f'(x)-f'(a))(a-x)^{1-\be}}{\Ga (2-\be)},
\end{equation}
so that
\begin{equation}
\label{eqRLCder2left}
D^{\be}_{a-\star}f(x)=D^{\be}_{a-}[f-f(a)-f'(a)(.-a)](x)
=D^{\be}_{a-}f(x)-\frac{f(a)(a-x)^{-\be}}{\Ga (1-\be)}+\frac{f'(a)(a-x)^{1-\be}}{\Ga (2-\be)}.
\end{equation}

For smooth bounded integrable functions the left fractional derivatives in generator form become
\begin{equation}
\label{eqRLdergenleft1}
\frac{d^{\be}}{d(-x)^{\be}}f(x)=D^{\be}_{\infty-}f(x)=D^{\be}_{\infty-\star}f(x)
=\frac{1}{\Ga (-\be)} \int_0^{\infty}\frac{f(x+z)-f(x)}{z^{1+\be}}dz,
\end{equation}
\begin{equation}
\label{eqRLdergenleft2}
\frac{d^{\be}}{d(-x)^{\be}}f(x)=D^{\be}_{\infty-}f(x)=D^{\be}_{\infty-\star}f(x)
=\frac{1}{\Ga (-\be)} \int_0^{\infty}\frac{f(x+z)-f(x)-f'(x)z}{z^{1+\be}}dz,
\end{equation}
for $\be \in (0,1)$ and $\be \in (1,2)$ respectively.

It is straightforward to see that the pairs of operators \eqref{eqRLder0inf}, \eqref{eqRLdergenleft1}
and \eqref{eqRLder2}, \eqref{eqRLdergenleft2} are dual in the sense that
\[
\left(\frac{d^{\be}}{dx^{\be}}f, g\right) =\left(f, \frac{d^{\be}}{d(-x)^{\be}}g\right)
\]
for $\be\in (0,1)\cup (1,2)$ and sufficiently regular functions $f,g$, where the pairing
$(f,g)$ denotes of course the usual $L^2$-product: $(f,g)=\int f(x) g(x) dx$. This fact also
justifies the notation $d^{\be}/d(-x)^{\be}$, since for $\be=1$
the operators $d/dx$ and $-d/dx=d/d(-x)$ are dual.

\section{Probabilistic interpretation and extensions of RL and Caputo derivatives}
\label{secprointfrac}

\subsection{The case  $\be \in (0,1), d=1$}

It is well known that the operators $-d^{\be}/d(-x)^{\be}$ and  $d^{\be}/d(-x)^{\be}$
from \eqref{eqRLdergenleft1} and \eqref{eqRLdergenleft2} respectively are the
generators of stable L\'evy motions without negative jumps, the (annoying) discrepancy
 in the sign reflects the fact that $\Ga (-\be)<0$ for $\be \in (0,1)$ and $\Ga (-\be)>0$ for $\be\in (1,2)$.
  In particular, for $\be \in (0,1)$, the corresponding L\'evy process is a stable subordinator
  (an increasing process). Similarly, the  operators $-d^{\be}/dx^{\be}$ and  $d^{\be}/dx^{\be}$
  for $\be \in (0,1)$ and $\be \in (1,2)$ respectively
generate stable L\'evy motions without positive jumps, which can be obtained by inversion
of the processes generate by $-d^{\be}/d(-x)^{\be}$ and  $d^{\be}/d(-x)^{\be}$.

Let us now look specifically at the decreasing process $X(t)$ generated by $A=-d^{\be}/dx^{\be}$, $\be \in (0,1)$.
Let us modify it by forbidding it (interrupting on an attempt) to cross a boundary $x=a$ with an $a\in \R$,
that is, all jumps aimed to land to the left of the chosen barrier-point $a$ are forced to land exactly at $a$.
Analytically, this procedure means changing the generator $A$ to
\[
A_{a+\star}f(x)= -\frac{1}{\Ga (-\be)} \int_0^{x-a}\frac{f(x-z)-f(x)}{z^{1+\be}}\, dz
-\frac{1}{\Ga (-\be)} \int_{x-a}^{\infty}\frac{f(a)-f(x)}{z^{1+\be}}\, dz, \quad x>a,
\]
which rewrites as
\begin{equation}
\label{eqgeninterbeta01}
A_{a+\star}f(x)= -\frac{1}{\Ga (-\be)} \int_0^{x-a}\frac{f(x-z)-f(x)}{z^{1+\be}}\, dz
+\frac{f(a)-f(x)}{\Ga (1-\be)(x-a)^{\be}}\, dz, \quad x>a.
\end{equation}
In this expression we recognize the Caputo derivative \eqref{eqdefCder}, with inverted sign.
Killing the process at the boundary-point $x=a$ means analytically to set $f(a)=0$, in which case
\eqref{eqgeninterbeta01} turns to RL fractional derivative (with inverted sign).

We conclude that the transition from the free derivative $d^{\be}/dx^{\be}$ to the Caputo right derivative at $a$
is a particular case of the procedure of interrupting a decreasing process on crossing a boundary. Namely,
let an operator
\[
Af(x)= \int_0^{\infty} (f(x-y)-f(x)) \nu (x, dy)
\]
with a kernel $\nu(x, .)$ on $(0,\infty)$ such that
\begin{equation}
\label{eqLevycondboundedvar}
\sup_x \int_0^{\infty} \min (1, |y|) \nu (x, dy) <\infty
\end{equation}
generate a decreasing Feller process. Then the corresponding process interrupted (and hence stopped)
at a boundary $x=a$ has the generator
\begin{equation}
\label{eqgeninterbeta01genright}
A_{a+\star}f(x)= \int_0^{x-a}  (f(x-y)-f(x)) \nu (x, dy)
+(f(a)-f(x))\int_{x-a}^{\infty} \nu (x, dy), \quad x>a.
\end{equation}

Similarly,  the transition from the free derivative $d^{\be}/d(-x)^{\be}$ to the Caputo left derivative at $a$
is a particular case of the procedure of interrupting an increasing process on crossing a boundary, that is,
the transition from a process on $\R$ generated by
\[
Af(x)= \int_0^{\infty} (f(x+y)-f(x)) \nu (x, dy)
\]
to the process on the interval $(-\infty,a]$ generated by
\begin{equation}
\label{eqgeninterbeta01genleft}
A_{a-\star}f(x)= \int_0^{a-x}  (f(x+y)-f(x)) \nu (x, dy)
+(f(a)-f(x))\int_{a-x}^{\infty} \nu (x, dy), \quad x<a.
\end{equation}

From this point of view, the natural extension of this procedure to the general processes of bounded variation
generated by the operators
\begin{equation}
\label{eqgenerorderone}
Af(x)= \ga (x) f'(x)+ \int_{-\infty}^{\infty} (f(x+y)-f(x)) \nu (x, dy)
\end{equation}
with a kernel $\nu(x, .)$ on $\R\setminus \{0\}$ such that
\begin{equation}
\label{eqLevycondboundedvartwoside}
\sup_x \int_{-\infty}^{\infty} \min (1, |y|) \nu (x, dy) <\infty
\end{equation}
is a transition to the process in a given interval $[a,b]$
with arbitrary $-\infty \le a < b \le \infty$
with jumps interrupted on an attempt to cross the boundary (i.e. all jumps aiming to land outside $[a,b]$ are forced to
land on its nearest point), that is the process generated by the operator
\[
A_{[a,b]\star}f(x)= \ga (x) f'(x) +\int_{a-x}^{b-x}  (f(x+y)-f(x)) \nu (x, dy)
\]
\begin{equation}
\label{eqgeninterbeta01gen}
+(f(b)-f(x))\int_{b-x}^{\infty} \nu (x, dy)+ (f(a)-f(x))\int^{a-x}_{-\infty} \nu (x, dy), \quad x\in (a,b).
\end{equation}

To deal with fractional derivatives we are mostly interested in the motion inside $[a,b]$.
However, for the analysis it is convenient to be able to have the corresponding process
extended to all $\R$. This can be done in various ways, though two natural approaches
can be distinguished. In the first one all jumps are supposed to be restricted to jump in the direction of $D$ only,
for instance one can choose the extension of $A_{[a,b]\star}$ to all $x$ is given by the expression
\[
A_{[a,b]\star}f(x)= \ga (x) f'(x) +\1_{x< b}\int_0^{\infty}  (f[(x+y)\wedge b]-f(x)) \nu (x, dy)
\]
\begin{equation}
\label{eqgeninterbeta01gen1}
+\1_{x>a}\int_{-\infty}^0  (f[(x+y)\vee a]-f(x)) \nu (x, dy).
\end{equation}
In the second approach we stick to the idea of interruption on crossing the boundary, so that $D$ and its complement
are treated symmetrically. To present this approach in a proper generality assume a finite set $B=\{b_1 < \cdots < b_k\}$ or
a countable set $B=\{b_i, i\in \N\}$, with $b_i <b_{i+1}$ for any $i$, is chosen. For instance, in the above setting
$B$ is a two-point set $B=\{a,b\}$. For any $x\in \R$ let us now define $b_+(x)$ and $b_-(x)$ as the nearest point
of $B$ to the right and to the left of $x$ respectively ($x$ excluded in both cases even if $x\in B$).
Then the modification of the process on $\R$ generated by \eqref{eqgenerorderone}, with jumps interrupted
 on crossing $B$ (think of $B$ as a set of road blocks or check points placed to control the free motion given by $A$)
 can be specified by the generator
\begin{equation}
\label{eqgeninterbeta01genext}
A_{B\star}f(x)= \ga (x) f'(x) + \int_{-\infty}^{\infty}  (f([(x+y)\wedge b_+(x)]\vee b_-(x))-f(x)) \nu (x, dy).
\end{equation}
Clearly for $x\in (b_i,b_{i+1})$ this generator coincides with \eqref{eqgeninterbeta01gen},
where $a=b_i, b=b_{i+1}$. In order to study the fractional differential equations in intervals
we can choose to work with the processes given by \eqref{eqgeninterbeta01gen1} or \eqref{eqgeninterbeta01genext},
as their behavior until they reach the boundary $\pa D$ is identical. For more general domains $D$
these different extensions can lead of course to different problems.


In \cite{Ko10}, \cite{Ko11} the author referred to generators of type \eqref{eqgenerorderone}
as to the generators of order at most one aiming to stress that they can be considered as fully mixed
fractional derivatives of order not exceeding one.

If we kill the process generated by $A$ at the boundary, the generator \eqref{eqgeninterbeta01gen} turns to
\begin{equation}
\label{eqgeninterbeta01genkil}
A_{[a,b]}f(x)= \int_{a-x}^{b-x}  (f(x+y)-f(x)) \nu (x, dy)
-f(x)\left[ \int_{b-x}^{\infty} \nu (x, dy)+ \int^{a-x}_{-\infty} \nu (x, dy)\right],
\end{equation}
for $x\in (a,b)$, which represents the corresponding extension of RL fractional derivative.

Therefore, from probabilistic point of view, the natural extension of the
basic linear fractional equation
\begin{equation}
\label{eqbaslinfraceq}
D^{\be}_{a+\star}f(x)=-\la f(x), \quad x>a,
\end{equation}
with initial condition $f_a=f(a)$ and $\la >0$ given, is the problem of
finding $f$ on $[a,\infty)$ satisfying
\[
A_{a+\star}f(x)= \la f(x), \quad f(a)=f_a,
\]
with $A_{a+\star}$ given by \eqref{eqgeninterbeta01genright}, that is the equation
\begin{equation}
\label{eqgeninterbeta01genrightlineq}
\int_0^{x-a}  (f(x-y)-f(x)) \nu (x, dy)
+(f(a)-f(x))\int_{x-a}^{\infty} \nu (x, dy)=\la f(x), \quad f(a)=f_a.
\end{equation}
This equation includes the extensions of \eqref{eqbaslinfraceq} with various
mixed fractional derivatives, that is, the problem
\begin{equation}
\label{eqixderlin}
\sum_j \om_j D^{\be_j}_{a+\star}f(x)=-\la f(x), \quad f(a)=f_a,
\end{equation}
with some finite collection of numbers $\om_j>0$ and $\be_j\in (0,1)$
or even more exotic versions with $\om_j$ or $\be_j$ being functions of $x$.
Mixed derivatives and related fractional equations are  actively studied recently by analytical methods,
see e. g. \cite{EideKoch04}, \cite{PskhuRusBook}, \cite{Pskhu11}, \cite{HiLo12}.

Moreover, it is now natural to formulate the two-sided version of \eqref{eqbaslinfraceq}
as the equation
\begin{equation}
\label{eqbaslinfraceqtwoside}
D^{\be}_{a+\star}f(x)+D^{\be}_{b-\star}f(x)=-\la f(x), \quad f(a)=f_a, f(b)=f_b
\end{equation}
on the interval $x\in [a,b]$, which represents the simplest case of a more general equation
\begin{equation}
\label{eqbaslinfracest0}
A_{[a,b]\star}f(x)=\la f(x), \quad f(a)=f_a, f(b)=f_b
\end{equation}
with $A_{[a,b]\star}$ given by \eqref{eqgeninterbeta01gen1}. In particular, the mixed-derivatives
extension of \eqref{eqbaslinfraceqtwoside} is the problem
\begin{equation}
\label{eqixderlintwoside}
-\sum_{j=1}^k \om_j D^{\be_j}_{a+\star}f(x)-\sum_{j=1}^k \ga_j D^{\be_j}_{b-\star}f(x)=\la f(x), \quad f(a)=f_a, f(b)=f_b,
\end{equation}
with some collection of numbers (or, more generally, functions) $\om_j>0, \ga_j>0, \be_j\in (0,1)$ and $\la >0$
extending for instance the problem considered in \cite{GarGuiMaiPag}.
The well-posedness for this problem is covered by Theorem \ref{theorderwelpos1} below.

As solutions to the equation $D_{a+\star}f(x)=g(x)$ with a given $g$ and initial condition $f(a)=a$
can be given in terms of fractional integrals, the two-sided analog of fractional integral becomes
the solution to the problem
\begin{equation}
\label{eqtwosidedfracint}
D^{\be}_{a+\star}f(x)+D^{\be}_{b-\star}f(x)=-g(x), \quad f(a)=f_a, f(b)=f_b,
\end{equation}
with given $f_a,f_b$ and $g$ on $[a,b]$, which again represents the simplest case of a more general problem
\begin{equation}
\label{eqbaslinfracest}
A_{[a,b]\star}f(x)=g(x), \quad f(a)=f_a, f(b)=f_b,
\end{equation}
or with the linear term included
\begin{equation}
\label{eqbaslinfracestlin}
A_{[a,b]\star}f(x)=\la f(x)+ g(x), \quad f(a)=f_a, f(b)=f_b,
\end{equation}
suggesting a new class of extensions of the notion of fractional integral,
which is alternative to a more classical one, see \cite{Kir14}, which is  
based on the variations of special functions used as the kernels 
(which is natural from an analytic perspective).

\begin{remark}
\label{remRLClink}
Notice that equation \eqref{eqbaslinfracestlin} can be rewritten in terms of RL derivative. Namely, let $v(x)$ be a smooth
function on $[a,b]$ with $v(a)=f_a, v(b)=f_b$. Then for the function $\phi=f-v$ problem \eqref{eqbaslinfracestlin} rewrites as
\begin{equation}
\label{eqbaslinfracestlin1}
A_{[a,b]\star}\phi =\la \phi+ g - A_{[a,b]\star}v +\la v, \quad \phi(a)=\phi(b)=0,
\end{equation}
which is equivalent to
\begin{equation}
\label{eqbaslinfracestlin2}
A_{[a,b]}\phi =\la \phi+ \tilde g, \quad \phi(a)=\phi(b)=0, \quad \tilde g= g+(\la - A_{[a,b]\star})v.
\end{equation}
\end{remark}

Particular examples of boundary value problems with fractional derivatives are actively studied, mostly
by analytical techniques, see e. g. \cite{Kadir14}, \cite{KiSrTr06} or \cite{MeerNaVe09} and references therein,
for distributed or mixed derivatives see also \cite{MaiMuPaGo}, \cite{GorLuchSto} and \cite{UmGor}.

\subsection{The case  $\be \in (0,1), d>1$}

Let us now turn to the multidimensional extension of this interruption procedure.
The analog of RL derivative arising from a process in $\R^d$ and a domain $D\subset \R^d$ is the generator of
 the process killed on leaving $D$. For Caputo version this is more subtle, as we have to specify
a point where a jump crosses the boundary. Below we choose the most natural model assuming that a trajectory
of a jump follows  shortest path (a straight line in Euclidean case).
Suppose $A$ is a generator of a Feller process $X_t(x)$ in $\R^d$ with the generator of type
\begin{equation}
\label{eqbasorder1d}
Af(x)= (\ga (x), \nabla) f(x) +\int_{\R^d} (f(x+y)-f(x)) \nu (x, dy)
\end{equation}
with a kernel $\nu(x, .)$ on $\R^d\setminus \{0\}$ such that
\begin{equation}
\label{eqLevycondboundedvartwoside}
\sup_x \int_{\R^d} \min (1, |y|) \nu (x, dy) <\infty,
\end{equation}
that is, in the terminology of \cite{Ko10}, \cite{Ko11}, a generator or order at most one.

Let $D$ be an open convex subset of $\R^d$ with boundary $\pa D$ and closure $\bar D$.
For $x\in \R^d$ and a unit vector $e$ let $L_{x,e}=\{x+\la e, \la \ge 0\}$ be the ray drawn from
$x$ in the direction $e$. For $x\in \R^d$, let
\[
D(x)=\{y\in \R^d : L_{x,y/|y|}\cap \bar D\neq \emptyset \}.
\]
In particular,
$D(x)=\bar D$ for all $x\in D$. Furthermore, for $y\in D(x)$, let
\[
\la (x,y/|y|)=\max \{ R>0: x+R y/|y| \in \bar D\}.
\]

Let us introduce the restriction function $R_D(x,y)$, $x\in \R^d$, $y\in D(x)$, by the formula
\begin{equation}
\label{eqdefprojonbound}
R_D(x,y) =\left\{
\begin{aligned}
& x+y, \quad \text{if} \, |y| \le \la (x, y/|y|) \\
& x+\la (x, y/|y|)y/|y|, \quad \text{if} \,  |y| \ge \la (x, y/|y|)
\end{aligned}
\right.
\end{equation}
The process $X_t(x)$ with jumps interrupted on crossing $\pa D$ when jumping from inside $D$
and forced to jump in the direction of $\bar D$ when jumping from outside of $D$
can be defined by the generator
\begin{equation}
\label{eqgeninterbeta01genmultdimC}
A_{D\star}f(x)=  (\ga (x), \nabla) f(x)+ \int_{D(x)}  [f(R_D(x,y))-f(x)] \nu (x, dy),
\end{equation}
which represents a multidimensional extension of the Caputo boundary operator in $D$
arising from \eqref{eqgeninterbeta01gen1}.

To define a multidimensional analog of \eqref{eqgeninterbeta01genext} let us assume
more generally that $\tilde D$ is an arbitrary open subset of $\R^d$,
so that $B=\R^d\setminus \tilde D$ is closed. Let us define
\[
\tilde \la (x,y/|y|)=\min \{ R >0: x+R y/|y| \in B\}
\]
for any $y$, with the convention that $\tilde \la (x,y/|y|)=\infty$ if the ray $L_{x,y/|y|}$ does not intersect $B$ at all,
and let the projection-on-the-boundary function $\tilde R_B(x,y)$ be defined for all $x,y\in \R^d$
by the same formula \eqref{eqdefprojonbound},
but with $\tilde \la$ instead of $\la$.
Then the analog of \eqref{eqgeninterbeta01genext}, that is
the modification of the process on $\R^d$ generated by \eqref{eqbasorder1d}, obtained
by interrupting jumps on an attempt to cross $B$, is specified by the generator
\begin{equation}
\label{eqgeninterbeta01genmultdimC1}
A_{\tilde D\star}f(x)=  (\ga (x), \nabla) f(x)+ \int_{\R^d}  [f(\tilde R_D(x,y))-f(x)] \nu (x, dy).
\end{equation}
For a convex domain $D$ the operators \eqref{eqgeninterbeta01genmultdimC} and
\eqref{eqgeninterbeta01genmultdimC1} with $B=\pa D$ and $\tilde D=\R^d\setminus B$
 coincide for $x\in D$, so that when one is interested
in the random motion inside $\bar D$ one can work with either of the processes generated
by  \eqref{eqgeninterbeta01genmultdimC} or \eqref{eqgeninterbeta01genmultdimC1}
and stopped on the boundary of $D$.

As pointed out above, the process $X_t(x)$ killed on the boundary,
which represents a multidimensional extension of the RL boundary operator in $D$
arising from \eqref{eqbaslinfracestlin2}, is specified by the generator
\begin{equation}
\label{eqgeninterbeta01genmultdimRL}
A_{D}f(x)= \int_{\R^d}  [f(x+y)\1_{x+y \in \bar D}-f(x)] \nu (x, dy), \quad x\in D.
\end{equation}

\begin{remark}
In dimension $d=1$ the projection $z\to R_{(a,b)}(x,z)=[z\wedge a]\vee b$
of the real numbers to the interval $[a,b]$ (used in \eqref{eqgeninterbeta01gen1})
does not depend on $x$ and is clearly the most natural one. In higher dimensions
one can imagine several reasonable extensions. Our choice used above was meant to
 preserve the direction of a jump. Another reasonable choice could be the definition of the
 projection $R_D(x,z)$ as the point on $\bar D$ nearest to $z$ (both choices coincide in $d=1$).
This would lead to a different multi-dimensional extension of the Caputo derivative.
\end{remark}

Assuming for simplicity that the kernel $\nu$ has
 a density, $\nu(x;y)$, with respect to Lebesgue measure, consider the following three basic examples.

If $D$ is a half space
\begin{equation}
\label{eqdefsemisp}
D=D_b=\{(x_1,x_2)\in \R^{d+1}: x_1 < b, x_2 \in \R^d\},
\end{equation}
then
\[
R_D(x,y)=\tilde R_D(x,y)= \left(b, x_2 +\frac{b-x_1}{y_1} y_2 \right), \quad x_1 < b \le x_1+ y_1,
\]
\[
A_{D\star}f(x)= \int_{\R^d} dy_2 \int_{-\infty}^{b-x_1} dy_1 \nu (x;y)[f(x+y)-f(x)]
\]
\begin{equation}
\label{eqCapsemisp}
+\int_{\R^d} dy_2 \int_{b-x_1}^{\infty} dy_1 \nu (x;y)[f\left(b, x_2 +\frac{b-x_1}{y_1} y_2 \right)-f(x)].
\end{equation}

If $D$ is a band
\begin{equation}
\label{eqdefband}
D=D_{(a,b)}=\{(x_1,x_2)\in \R^{d+1}: a < x_1 < b, x_2 \in \R^d\},
\end{equation}
then, for $x_1\in (a,b)$,
\[
A_{D\star}f(x)= \int_{\R^d} dy_2 \int_{a-x_1}^{b-x_1} dy_1 \nu (x;y)[f(x+y)-f(x)]
\]
\[
+\int_{\R^d} dy_2 \int_{b-x_1}^{\infty} dy_1 \nu (x;y)[f\left(b, x_2 +\frac{b-x_1}{y_1} y_2 \right)-f(x)]
\]
\begin{equation}
\label{eqCapband}
+\int_{\R^d} dy_2 \int_{-\infty}^{a-x_1} dy_1 \nu (x;y)[f\left(a, x_2 +\frac{a-x_1}{y_1} y_2 \right)-f(x)].
\end{equation}

If $D$ is the unit ball $D=\{x:|x|<1\}$ in $\R^d$, then, for $x\in D$,
\[
R_D(x,y)=\tilde R_D(x,y)= x+\la (x,y) y, \quad \la (x,y)
=\frac{1}{|y|}\left[\sqrt{\left(x,\frac{y}{|y|}\right)^2+1-|x|^2}-\left(x,\frac{y}{|y|}\right)\right],
\]
\begin{equation}
\label{eqCapball}
A_{D\star}f(x)= \int_{\R^d} dy \, \nu (x;y) \left(\1_{\la(x,.) \ge 1} [f(x+y)-f(x)]
+ \1_{\la(x,.) <1} [f(x+\la (x,y) y)-f(x)]\right).
\end{equation}

\subsection{The case $\be \in (1,2)$}

Let us now look at the derivatives of order $\be \in (1,2)$. Interrupting the L\'evy process with
only negative jumps generated by \eqref{eqRLder2} on crossing the boundary $\{x=a\}$ means changing its generator to
the operator
\begin{equation}
\label{eqRLder2intera}
\tilde D^{\be}_{a+}f(x)=\frac{1}{\Ga (-\be)} \int_0^{x-a}\frac{f(x-z)-f(x)+f'(x)z}{z^{1+\be}}dz
+\frac{1}{\Ga (-\be)} \int_0^{x-a}\frac{f(a)-f(x)+f'(x)z}{z^{1+\be}}dz,
\end{equation}
which rewrites as
\[
\tilde D^{\be}_{a+}f(x)=\frac{1}{\Ga (-\be)} \int_0^{x-a}\frac{f(x-z)-f(x)+f'(x)z}{z^{1+\be}}dz
\]
\begin{equation}
\label{eqRLder2intera1}
+\frac{(f(x)-f(a))(x-a)^{-\be}}{\Ga (1-\be)}+\frac{\be f'(x)(x-a)^{1-\be}}{\Ga (2-\be)}.
\end{equation}
This expression differs from the Caputo derivative \eqref{eqCder1} by the term containing $f'(a)$:
\begin{equation}
\label{eqRLder2intera1}
D^{\be}_{a+\star}f(x)=\tilde D^{\be}_{a+}f(x)-\frac{f'(a)(x-a)^{1-\be}}{\Ga (2-\be)},
\end{equation}
which one could expect as operator \eqref{eqCder1} does not have a structure that allows one to interpret
it as a generator for a Markov process precisely because of this term containing $f'(a)$ (recall that
$\Ga (-\be)>0$ and $\Ga (1-\be)<0$ here).

In order to see the meaning of this correcting term let us observe that if $f\in C^2[a,\infty)$, then, up to terms tending to zero,
$\tilde D^{\be}_{a+}f(x)$ behaves like
\[
f'(x)(x-a)^{1-\be}\left[\frac{1}{\Ga (1-\be)}+\frac{\be}{\Ga (2-\be)}\right]=\frac{f'(x) (x-a)^{1-\be}}{\Ga (2-\be)},
\]
as $x\to a$, and thus tends to $\pm \infty$ if $f'(a)$ is positive or negative respectively.
Thus subtraction of the term containing $f'(a)$ in \eqref{eqRLder2intera1} is a regularization of
$\tilde D^{\be}_{a+}$ that makes it finite on smooth functions.

On the other hand, killing the process generated by \eqref{eqRLder2intera1} at the boundary point $x=a$
means setting $f(a)=0$ and then \eqref{eqRLder2intera1} turns exactly into the Riemann-Liouville
derivative
\begin{equation}
\label{eqRLder2intera1kil}
D^{\be}_{a+}f(x)=\frac{1}{\Ga (-\be)} \int_0^{x-a}\frac{f(x-z)-f(x)+f'(x)z}{z^{1+\be}}dz
+\frac{f(x)(x-a)^{-\be}}{\Ga (1-\be)}+\frac{\be f'(x)(x-a)^{1-\be}}{\Ga (2-\be)},
\end{equation}
precisely as in the case $\be \in (0,1)$.

Extension of this procedure is thus clear. Namely, starting with a Feller process with negative jumps, say with the generator
\begin{equation}
\label{eqlevynegjump}
Af(x)=\int_0^{\infty}[f(x-z)-f(x)+f'(x)z]\nu (x,dz)
\end{equation}
such that
\begin{equation}
\label{eqlevynegjumpunboundvar}
\sup_x \int_0^{\infty} z^2 \nu (x, dz)<\infty, \quad \int_0^{\infty} z \nu (x,dz) =\infty
\end{equation}
one can form the corresponding process with jumps interrupted at $a$ as the process generated by
\[
\tilde A_{a+} f(x)=\int_0^{x-a}[f(x-z)-f(x)+f'(x)z]\nu (x,dz)
\]
\begin{equation}
\label{eqlevynegjumpinter}
+(f(a)-f(x)) \int_{x-a}^{\infty}\nu (x,dz)
+f'(x) \int_{x-a}^{\infty} z \nu (x,dz).
\end{equation}
Its main term of asymptotics as $x\to a$ is
\[
f'(x) \int_{x-a}^{\infty} (z-(x-a)) \nu (x,dz),
 \]
which is unbounded unless $f'(a)=0$. Thus the analog of the Caputo fractional derivative is obtained by subtracting
this 'infinity':
\begin{equation}
\label{eqlevynegjumpinterreg}
A_{a+\star} f(x)= \tilde A_{a+} f(x)-f'(a) \int_{x-a}^{\infty}  (z-(x-a)) \nu (a,dz).
\end{equation}

Moreover, it implies the following limiting behavior for $f\in C^2[a,\infty)$:
\begin{equation}
\label{eqlevynegjumpinterreg1}
\begin{aligned}
& \lim_{x\to a_+} A_{a+\star} f(x)=0, \\
& \lim_{x\to a_+} \tilde A_{a+} f(x)=0, \, \text{if} \, f'(a)=0.
\end{aligned}
\end{equation}

\begin{remark} Notice that generator \eqref{eqlevynegjumpinter} has variable coefficients even if underlying process was a
L\'evy process. Hence it is not directly clear whether \eqref{eqlevynegjumpinter} generates a well defined process.
This issue will be addressed in the next sections.
\end{remark}

Similarly one defines $A_{a-\star}$ for Feller processes with positive jumps.
More generally, for any Feller process on $\R$ generated by the operator
\begin{equation}
\label{eqlevyneg}
Af(x)=G(x)f''(x)+\ga (x) f'(x) +\int_{\R}[f(x+z)-f(x)-f'(x)z\chi (z)]\nu (x,dz),
\end{equation}
where $\chi$ is some mollifier (that traditionally is taken as either $\1_{|z|\le 1}$ or as $1/(1+z^2)$),
one defines the corresponding process with jumps interrupted on crossing an interval $[a,b]$
as the process generated by the operator
\[
\tilde A_{[a,b]}f(x)=G(x)f''(x)+\ga (x)f'(x) +\int_{a-x}^{b-x}[f(x+z)-f(x)-f'(x)z\chi (z)]\nu (x, dz)
\]
\begin{equation}
\label{eqlevyneginterab}
+(f(b)-f(x))\int_{b-x}^{\infty} \nu (x,dz)+ (f(a)-f(x))\int^{a-x}_{-\infty} \nu (x,dz)
-f'(x) \int_{\R \setminus (a-x,b-x)} z \chi (z) \nu (x,dz)
\end{equation}
for $x\in (a,b)$ (when such process is well defined), or for arbitrary $x$
 \[
\tilde A_{[a,b]}f(x)=G(x)f''(x)+\ga (x)f'(x) +\1_{x<b}\int_0^{\infty}\left(f[(x+z)\wedge b]-f(x)-f'(x)z\chi (z)\right)\nu (x,dz)
\]
\begin{equation}
\label{eqlevyneginterab1}
+\1_{x>a}\int_{-\infty}^0(f([(x+z)\vee a]-f(x)-f'(x)z\chi (z))\nu (x,dz).
\end{equation}

The corresponding analog of the Caputo derivative is obtained by subtracting
 the singularity at the boundary points, that is, if
\begin{equation}
\label{eqlevynegjumpunboundvar1}
 \int_{-1}^0 |z| \nu (x,dz) =\infty, \quad \int_0^1 z \nu (x,dz) =\infty,
\end{equation}
then, for $x\in (a,b)$,
\begin{equation}
\label{eqlevyneginterabreg}
A_{[a,b]\star}f(x)=\tilde A_{[a,b]}f(x)-f'(b) \int_{b-x}^{\infty} [(b-x)-z \chi (z)]\nu(x, dz)
-f'(a)\int_{-\infty}^{a-x} [(a-x)-z \chi (z)]\nu (x,dz).
\end{equation}
The analog of the Riemann-Liouville derivative is obtained from $\tilde A_{[a,b]}f(x)$ by setting the boundary
values of $f$ to zero yielding the operator of the processes generated by $\tilde A_{[a,b]}f(x)$, but killed
at the boundary:
\[
A_{[a,b]}f(x)=G(x)f''(x)+\ga (x) f'(x) +\int_{a-x}^{b-x}[f(x+z)-f(x)-f'(x)z\chi (z)]\nu (x,dz)
\]
\begin{equation}
\label{eqlevyneginterab}
-f(x) \int_{\R \setminus (a-x,b-x)} \nu (x,dz)
-f'(x) \int_{\R \setminus (a-x,b-x)} z \chi (z) \nu (x,dz).
\end{equation}

But what is the relation between equations involving $A_{[a,b]\star}$ and $\tilde A_{[a,b]}$? The point is that
if $\tilde A_{[a,b]}f=g$ on $[a,b]$ for some bounded $g$, then necessarily $f'(a)=f'(b)=0$ (as otherwise
$\tilde A_{[a,b]}$ would be unbounded for $x\to a$ or $x\to b$) and hence
$\tilde A_{[a,b]}f=A_{[a,b]\star}$, so that at least classical solutions for the problems
\begin{equation}
\label{eqbaslinfracestrep2}
\tilde A_{[a,b]}f(x)=\la f(x) +g(x), \quad f(a)=f_a, f(b)=f_b
\end{equation}
also solve the problem
\begin{equation}
\label{eqbaslinfracestrep1}
A_{[a,b]\star}f(x)=\la f(x) +g(x), \quad f(a)=f_a, f(b)=f_b.
\end{equation}
The problem \eqref{eqbaslinfracestrep2} can be naturally settled probabilistically.

\begin{remark} The subtraction of singularity in definition \eqref{eqlevyneginterabreg} makes it dependent
on representation \eqref{eqlevyneg}. This is however not very essential for solving the corresponding
fractional differential equations, because operator $\tilde A$ is not representation dependent.
\end{remark}

Similarly in multidimensional case the analog of the Caputo derivative can be obtained
by subtracting the singularity from the generator of the process obtained from an initial one by restricting the jumps to
a chosen domain $\bar D$.
For instance, for a Feller process generated by the operator
\begin{equation}
\label{eqlevygenjump}
Af(x)=\int_{\R^d}[f(x+z)-f(x)-(f'(x),z)\chi (z)]\nu (x,dz),
\end{equation}


the process with jumps restricted to land on $\bar D$ (stopped-on-crossing-the-boundary process) for $D$ the half-space
\eqref{eqdefsemisp} has the generator

\[
\tilde A_{D}f(x)= \int_{\R^d} dy_2 \int_{-\infty}^{b-x_1} dy_1 \nu (x;y)[f(x+y)-f(x)-(f'(x), y)\chi (y)]
\]
\begin{equation}
\label{eqCapsemispbelarge}
+\int_{\R^d} dy_2 \int_{b-x_1}^{\infty} dy_1 \nu (x;y)[f\left(b, x_2 +\frac{b-x_1}{y_1} y_2 \right)-f(x)-(f'(x), y)\chi (y)],
\end{equation}
for $x\in D$,
which is a direct extension of \eqref{eqCapsemisp}. However, unlike \eqref{eqCapsemisp},
$\tilde A_Df(x)$ now diverges as $x$ tends to a boundary point if $\pa f/\pa x_1 (x)\neq 0$ there.
In fact, as one sees directly, the main term of $\tilde A_Df(x)$ as $(x_1,x_2) \to (b,x_2)$ is
\[
\frac{\pa f}{\pa x_1} (x_1,x_2) \int_{b-x_1}^{\infty} dy_1 \int_{\R^d} dy_2 \, (b-x_1-y_1) \nu (b,x_2;y_1,y_2).
\]
Hence the analog of the Caputo derivative is
\begin{equation}
\label{eqCapsemispbelargeCap}
A_{D\star}f(x)=\tilde A_{D}f(x) - \frac{\pa f}{\pa x_1} (b,x_2) \int_{b-x_1}^{\infty} dy_1 \int_{\R^d} dy_2 \, (b-x_1-y_1) \nu (b,x_2;y_1,y_2).
\end{equation}

\section{Basic well-posedness results and examples}
\label{secwelposedfrac}

\subsection{The case $\be \in (0,1), d=1$}

Let us start with the operator \eqref{eqgeninterbeta01gen} assuming for simplicity that the kernels $\nu$ have densities
with respect to Lebesgue measure. Therefore let
\begin{equation}
\label{eqgenjumpboundvar}
Af(x)=\ga (x) f'(x)+\int_{\R}  (f(x+y)-f(x)) \nu (x, y) \, dy.
\end{equation}
Choosing for definiteness extension \eqref{eqgeninterbeta01gen1}
 (of \eqref{eqgeninterbeta01gen} defined for $x\in (a,b)$) let
\[
A_{[a,b]\star}f(x)= \ga (x) f'(x) +\1_{x< b}\int_0^{\infty}  (f[(x+y)\wedge b]-f(x)) \nu (x, dy)
\]
\begin{equation}
\label{eqgeninterbeta01genrep}
+\1_{x>a}\int_{-\infty}^0  (f[(x+y)\vee a]-f(x)) \nu (x, dy),
\end{equation}
which can be represented as the sum of the generators of a decreasing process, an increasing process and a drift:
\[
A_{[a,b]\star}=A_{a+\star}+A_{b-\star}+\ga (.) \frac{d}{dx}
\]
with
\[
A_{a+\star}f(x)= \1_{x>a}\int_{-\infty}^0  (f[(x+y)\vee a]-f(x)) \nu (x, dy),
\]
\[
A_{b-\star}f(x)= \1_{x< b}\int_0^{\infty}  (f[(x+y)\wedge b]-f(x)) \nu (x, dy).
\]

Let $-\infty < a < b <\infty$.
Let us denote, as usual,  by $C[a,b]$ (resp. $C_{\infty}(\R)$) the Banach space of continuous
functions on $[a,b]$ (resp. on $\R$ vanishing at infinity),
by $C_{\infty}(-\infty,a]$ (resp. $C_{\infty}[a,\infty)$) the Banach space of continuous functions
on $(-\infty,a]$ (resp. $[a,\infty)$) vanishing at infinity, by $C^k_{\infty}(\R)$, $C^k_{\infty}(-\infty,a]$,
$C^k_{\infty}[a,\infty)$, $C^k[a,b]$ the subspaces of functions of the corresponding spaces $C_{\infty}(\R)$, $C_{\infty}(-\infty,a]$,
$C_{\infty}[a,\infty)$, $C[a,b]$, having derivatives up to order $k$ from $C_{\infty}(\R)$, $C_{\infty}(-\infty,a]$,
$C_{\infty}[a,\infty)$, $C[a,b]$, respectively.

Recall also that, for a domain $D \in \R^d$ with boundary $\pa D$ and a Markov process $X_x(t)$ in $\R^d$ or just in $D$,
a point $x\in \pa D$ is called regular if $\tau_D(x) \to 0$ in probability for $x\to x_0$, $x\in D$, where $\tau_D$ is the exit
time from $D$. We say that it is regular in expectation if $\E \tau_D(x) \to 0$  for $x\to x_0$, $x\in D$.

\begin{remark}
1. A point zero for $D=(0,\infty)$ is regular for a Brownian motion, but not regular in expectation, as
$\E \tau_D(x) =\infty$ for all $x\in D$.
2. We are using the notion of a regular point arising from the theory
of parabolic PDEs (see e. g. \cite{Fr} or \cite{Ko11}),
which is different from the corresponding notion used in the theory of L\'evy processes (see \cite{Kypbook}).
\end{remark}

\begin{theorem}
\label{theorderonecapder}
Assume that $\ga (x)$ is continuously differentiable with a bounded derivative and that
$\nu(x,y)$ is a continuous function of two variables, which is continuously differentiable
with respect to the first variable and has the following uniform bounds and tightness property
\begin{equation}
\label{eq1theorderonecapder}
\sup_x \int |y|\nu (x,y) \, dy <\infty,
\quad \sup_x \int |y|\, \left|\frac{\pa}{\pa x}\nu (x,y)\right| \, dy <\infty,
\end{equation}
and
\begin{equation}
\label{eq2theorderonecapder}
\lim_{\de \to 0} \sup_x \int_{|y|\le \de} |y|\nu (x,y) \, dy =0.
\end{equation}
(i) Then the operator $A_{a+\star}$ generates a Feller process on $[a,\infty)$
and a Feller semigroup on $C_{\infty}[a,\infty)$ with the invariant core $C^1_{\infty}[a,\infty)$,
 the operator $A_{b-\star}$ generates a Feller process on $(-\infty,b]$
and a Feller semigroup on $C_{\infty}(-\infty,b]$ with the invariant core $C^1_{\infty}(-\infty,b]$.

(ii) Moreover, the operator $A_{[a,b]\star}$ generates a Feller process on $\R$
and a Feller semigroup on $C_{\infty}(\R)$ and, if $\ga (x)=0$,
then $A_{[a,b]\star}$ generates also a Feller process on $[a,b]$
and a Feller semigroup on $C[a,b]$ with the invariant core $C^1[a,b]$.

(iii) If
\begin{equation}
\label{eq20theorderonecapder}
\int_{-\infty}^0 \min (|y|,\ep)\nu (a,y) \, dy >C\ep^r
\end{equation}
or
\begin{equation}
\label{eq200theorderonecapder}
\int_0^{\infty} \min (y,\ep)\nu (b,y) \, dy >C\ep^r
\end{equation}
for some $C>0$, $r \in (0,1)$, then the point $a$ (resp. $b$) is regular in expectation
for the first (resp. second) process in (i).

(iv) If $\ga (a)<0$ (resp. $\ga (b) >0$), then
then the point $a$ (resp. $b$) is a regular in expectation boundary point of the interval $(a,b)$
for the process generated by $A_{[a,b]\star}$.

(v) If $A_{[a,b]\star}$ is the operator on the l.h.s. of \eqref{eqixderlintwoside} and $\om_{j_0}>0, \ga_{j_0}>0$, where
$\be_{j_0}$ is the unique maximum of all $\be_j$, then the points $a$ and $b$ are regular in expectation
boundary points of $(a,b)$ for the process generated by $A_{[a,b]\star}$.

\end{theorem}

\begin{proof}
(i)
Notice that the operator $A_{[a,b]\star}$ is a bounded operator $C^1[a,b] \to C[a,b]$ such that
\begin{equation}
\label{eq20btheorderonecapder}
\lim_{x\to a_+} A_{a+\star}f(x)=0, \quad \lim_{x\to b_-} A_{b-\star}f(x)=0
\end{equation}
for $f\in C^1[a,b]$, as follows from
\begin{equation}
\label{eq20atheorderonecapder}
\lim_{\de \to 0} \de \int_{\R\setminus [-\de,\de]}\nu (x,y) \, dy =0.
\end{equation}

\begin{remark}
Equation \eqref{eq20atheorderonecapder} is  a consequence of the first bound in \eqref{eq1theorderonecapder}.
In fact, since
\[
\int_{\de}^1 y\nu (x,y) \, dy = \de F_x(\de) +\int_{\de}^1 F_x (y) \, dy,
\]
where  $F_x(y)=\int_y^1 \nu(x,z) \, dz$, it follows that the both terms on the r.h.s. of this equation
are uniformly bounded. Hence, the l.h.s. of this equation and the second term on the r.h.s. converge to
$\int_0^1 y\nu (x,u) \, dy$ implying  $\de F_x(\de) \to 0$ as $\de \to 0$.
\end{remark}

We shall follow the strategy of proof from Theorem 5.1.1 of \cite{Ko11}.
Let us now work for definiteness with $A_{b-\star}$ (other cases are dealt with analogously).
Differentiating (and using straightforward cancelations) yields
\[
\frac{d}{dx} A_{b-\star}f(x)= \int_0^{b-x}  \left[f'(x+y)-f'(x)\right] \nu (x, y) \, dy
-f'(x)\int_{b-x}^{\infty} \nu (x, y)\, dy
\]
\[
+\int_0^{b-x}  (f(x+y)-f(x)) \frac{\pa \nu}{\pa x}(x, y) \, dy
+ (f(b)-f(x))\int_{b-x}^{\infty} \frac{\pa \nu}{\pa x} (x, y)\, dy.
\]
Thus if $f$ solves the equation
\[
\dot f =A_{b-\star}f,
\]
then $g=f'$ (if exists) solves the equation
\[
\dot g =\int_0^{b-x}  \left[g(x+y)-g(x)\right] \nu (x, y) \, dy
-g(x)\int_{b-x}^{\infty} \nu (x, y)\, dy
\]
\begin{equation}
\label{eq3theorderonecapder}
+\int_0^{b-x} dy \int_0^y g(x+z) dz \frac{\pa \nu}{\pa x}(x, y)
+ \int_x^b g(z) dz \int_{b-x}^{\infty} \frac{\pa \nu}{\pa x} (x, y)\, dy.
\end{equation}

Let us introduce the approximation $A_{b-\star h}$ for our operator obtained by changing $\nu(x,y)$ to
$\nu_h(x,y)=\1_{|y|>h} \nu (x,y)$. For any $h$ the operator $A_{b-\star h}$
is bounded in $C(-\infty,b]$ and hence generates a conservative Feller semigroup $T_t^h$ there.
Moreover, the operator on the r.h.s. of \eqref{eq3theorderonecapder}
becomes also bounded when $\nu$ is replaced by $\nu_h$, so that this equation becomes well-posed in
$C(-\infty,b]$, so that $T_t^h$ is also a strongly continuous semigroup in $C^1(-\infty,b]$.

The key observation now is that  $T_t^h$ are bounded in $C^1(-\infty,b]$
uniformly in $h$, because
the first two terms on the r.h.s. of \eqref{eq3theorderonecapder} represent a conditionally positive operator
with a negative coefficients at $g$, which therefore generates a positivity preserving
contraction in $C(-\infty,b]$, and the last two terms are uniformly (in $h$) bounded operators.
Hence $(T_t^h f)'(x)$ (by $'$ we denote the derivative with respect to $x$)
are uniformly bounded for all $h \in (0,1]$
and $t$ from any compact interval whenever $f\in
C_{\infty}^1(-\infty,b]$. Therefore, writing
\[
(T_t^{h_1}-T_t^{h_2})f=\int_0^t
T_{t-s}^{h_2}(A_{b-\star h_1}-A_{b-\star h_2})T_s^{h_1}\, ds
\]
for arbitrary $h_1 >h_2$ and estimating
\[
|(A_{b-\star h_1}-A_{b-\star h_2})T_s^{h_1}f(x)| \le
 \int_{h_2 \le |y| \le h_1}
 |(T_s^{h_1}f)(x+y)-(T_s^{h_1}f)(x)| \nu (x,y)\, dy
\]
\[
 \le \int_0^{h_1}
 \| ( T_s^{h_1}f)' \| \nu (x,y) |y| \, dy=o(1) \|f\|_{C(-\infty,b]}, \quad h_1\to 0,
 \]
 yields
\begin{equation}
\label{eqshortdifference}
 \|(T_t^{h_1}-T_t^{h_2})f\|=o(1) t\|f\|_{C^1(-\infty,b]}, \quad h_1\to 0.
\end{equation}
Therefore the family
$T_t^hf$ converges to a family $T_tf$,  as $h \to 0$, which also forms a strongly continuous semigroup in
$C(-\infty,b]$.
Writing
\[
\frac{T_tf-f}{t}=\frac{T_tf-T_t^hf}{t}+\frac{T_t^hf-f}{t}
\]
and noting that by \eqref{eqshortdifference} the first term is of
order $o(1) \|f\|_{C^1(-\infty,b]}$ as $h\to 0$ allows one to conclude
that  $C^1(-\infty,b]$ belongs to the domain of the generator of the semigroup $T_t$ in
$C(-\infty,b]$.

 Applying to $T_t$ the procedure applied above to $T_t^h$ (differentiating
 the evolution equation with respect to $x$) shows that $T_t$
 defines also a strongly continuous semigroup in $C^1(-\infty,b]$, and hence $C^1(-\infty,b]$
 is an invariant core for $T_t$.

\begin{remark}
Notice that the semigroup $T_t$ extends also to the strongly continuous semigroup
on $C_{\infty}(\R^)$ generated by $A_{b-\star}$ with an invariant domain $C^1_{\infty}$,
because the right condition of \eqref{eq20btheorderonecapder} ensures a smooth gluing
with the value $A_{b-\star}fx)=0$ for $x>b$ for any $f\in C^1(\R)$.
\end{remark}

(ii) This is quite similar and is omitted.

(iii) To prove regularity of the boundary, we shall use the method of Lyapunov functions.
Namely, to show that, say, $b$ is regular for
the last process in (i), it is sufficient to find a continuous function $f$ in a neighborhood of $[a,b]$ such that
$f$ is differentiable for $x\in (a,b)$,
$f(b)=0$,  and for $x\in (c,b)$ with some $c \in (a,b)$ one has $f(x)>0$, $A_{[a,b]\star}f(x) <0$
(see Proposition 6.3.2 of \cite{Ko11}). As such function one can take $f_{\om}(x)=(b-x)^{\om}$ with some $\om \in (0,1)$.
In fact, clearly $f_{\om} (b)=0$, $f_{\om}(x)>0$ for $x<b$ and for $x$ approaching $b$ from the left,
$A_{[a,b]\star}f(x)$ is of order
\[
-(b-x)^{\om-1} \int_0^{\infty} \min (y,b-x) \nu (b,y) \, dy
\]
which tends to $-\infty$ as $x\to b$ for sufficiently small $\om$ under the assumptions of (iii).

(iv) Using the same Lyapunov function $f_{\om}(x)=(b-x)^{\om}$ one sees that the drift term always
dominates the jump part part of the generator.

(v) Proving regularity in a general case can be subtle. However, for the process from (v) the same Lyapunov function
 $f_{\om}(x)=(b-x)^{\om}$ yields the required result.

\end{proof}

Theorem \ref{theorderonecapder} allows one to solve equations involving $A_{[a,b]\star}$ or $A_{a \pm \star}$
by the standard techniques of stochastic analysis. Namely, let us denote by $X_x(t)$ the Markov process generated by
$A$ of \eqref{eqgenjumpboundvar} and by $X^*_x(t)[a,b]$ the Markov process on $[a,b]$ generated by $A_{[a,b]\star}$.

Let us stress that the process $X^*_x(t)[a,b]$ (which is interrupted but not stopped at the boundary)
is generated by $A_{[a,b]\star}$ defined on the domain $C^1_{\infty}(\R)$ (or else on the domain
$C^1[a,b]$ if $\ga (b)\le 0$ and $\ga (a) \ge 0$); the process $X^*_x(t)[a,b;stop]$ on $[a,b]$
obtained from $X^*_x(t)[a,b]$ by stopping at the boundary is generated by
$A_{[a,b]\star}$ defined on the domain, which is a subspace of $C^1[a,b]$ of functions $f$
such that $A_{[a,b]\star}f(a)=A_{[a,b]\star}f(b)=0$;
the process $X^*_x(t)[a,b;kill]$ killed at the boundary is generated by
$A_{[a,b]\star}$ defined on the domain, which is a subspace of $C^1[a,b]$ of functions $f$
such that $A_{[a,b]\star}f(a)=A_{[a,b]\star}f(b)=0$ and $f(a)=f(b)=0$. It is easily seen that
if $\int_{-\infty}^a \nu (a,y)dy =\infty$ and $\int_b^{\infty} \nu (b,y)dy =\infty$, the operator
$A_{[a,b]\star}$ generates a strongly continuous semigroup (of the process $X^*_x(t)[a,b]$
killed on the boundary) on the subspace $C_0[a,b]$
of $C[a,b]$ consisting of functions vanishing at $a$ and $b$.

Let $\tau$ denote the first exit time for $X^*_x(t)[a,b]$ or $X_x(t)$ from $(a,b)$:
\[
\tau =\inf \{t\ge 0: X^*_x(t)[a,b]\notin (a,b)\}=\inf \{t\ge 0: X_x(t)\notin (a,b)\}.
\]
Applying Dynkin's martingale to a function $f\in C^1[a,b]$ and Doob's optional sampling theorem
for the stopping time $\tau$ (see e. g. \cite{EK} for the presentation of these two basic
tools of stochastic analysis) yields
\[
f(x)=\E \left[ f(X^*_x(\tau)[a,b]) -\int_0^{\tau} (A_{[a,b]\star} f)(X^*_x(s)[a,b]) \, ds \right]
\]
Hence if $f$ is a solution to problem \eqref{eqbaslinfracest}, then
\begin{equation}
\label{eqsolvetwosidgen1}
f(x)= f(a) \P (X^*_x(\tau)[a,b]=a)+f(b) \P (X^*_x(\tau)[a,b]=b) - \E \int_0^{\tau} g (X^*_x(s)[a,b]) \, ds.
\end{equation}
Moreover, since the trajectories of $X_x(t)$ and $X^*_x(t)[a,b]$ coincide till time $\tau$ this can be expressed
entirely in terms of the process $X_x(t)$ as
\begin{equation}
\label{eqsolvetwosidgen2}
f(x)= f(a) \P (X_x(\tau)\le a)+f(b) \P (X_x(\tau) \ge b) - \E \int_0^{\tau} g (X_x(s)) \, ds.
\end{equation}
Introducing, as is common in the theory of L\'evy processes, the occupation-till-exit measure $H(x,dy)$ on $(a,b)$ by
\begin{equation}
\label{eqdefoccuptllexit}
H(x,B)=\E \int_0^{\tau} \1_B(X_x(s)) \, ds,
\end{equation}
allows one to rewrite \eqref{eqsolvetwosidgen2} as
\begin{equation}
\label{eqsolvetwosidgen3}
f(x)= f(a) \P (X_x(\tau)\le a)+f(b) \P (X_x(\tau) \ge b) - \int_a^b g (y) H(x,dy).
\end{equation}

Thus we obtain the following result.
\begin{theorem}
\label{theorderwelpos1}
(i) Under the assumptions of Theorem  \ref{theorderonecapder} problem \eqref{eqbaslinfracest}
can have at most one classical solution. If $\E \tau_{[a,b]}(x) <\infty$ for all $x\in D$,
the probabilities $\P (X_x(\tau)\le a)$, $\P (X_x(\tau)\ge b)$ are
functions from $C^1[a,b]$ and the measure $H(x,dy)$ is continuously weakly differentiable in $x$ (so that the integral
on the r.h.s. of \eqref{eqsolvetwosidgen3} belongs to $C^1[a,b]$ for any $g\in C[a,b]$),
then formula \eqref{eqsolvetwosidgen3} supplies the unique classical solution to problem \eqref{eqbaslinfracest}.
\end{theorem}

\begin{remark}
In the classical analysis of boundary-value problems for partial differential equations,
problems like the one in \eqref{eqbaslinfracest} are usually understood to mean that the main equation there holds
for all $x$ excluding the boundary. The necessity of this attitude is easily seen here.
Namely, if $f$ belongs to the domain of the generator of a stopped or killed process,
then the value of this generator on $f$ at a boundary point should vanish. Thus only for $g$ vanishing on the boundary
the solution to \eqref{eqbaslinfracest} can belong to the domain of the generator and
satisfy the main equation up to the boundary.
\end{remark}

Furthermore, as is known from stochastic analysis, formula \eqref{eqsolvetwosidgen3} makes sense as
a generalized solution under more general assumptions.
Not going into much detail, let us only menton one particular situation. Namely, one says
 that a continuous function $f(x)$ on $[a,b]$ is a generalized solution to
problem \eqref{eqbaslinfracest} with $g=0$, if $f$ belongs to the domain of the generator of the
semigroup $T_t^{stop}[a,b]$ of the stopped process $X^*_x(t)[a,b;stop]$ on $[a,b]$
(obtained by the closure of the operator $A_{[a,b]\star}$ defined on the domain,
which is a subspace of $C^1[a,b]$ of functions $f$
such that $A_{[a,b]\star}f(a)=A_{[a,b]\star}f(b)=0$) and satisfies $A_{[a,b]\star}f=0$,
or equivalently $T_t^{stop}[a,b]f=f$. The following fact is a consequence
of a general theory of boundary points (here the regularity of the boundary is crucial),
see e. g. Theorem 6.2.3 of \cite{Ko11} for detail.
\begin{theorem}
\label{theorderwelpos1a}
 Under the assumptions of Theorem  \ref{theorderonecapder},
formula \eqref{eqsolvetwosidgen3} supplies a unique generalized solution to problem \eqref{eqbaslinfracest}.
\end{theorem}

Furthermore, to solve problem \eqref{eqbaslinfracestlin} one utilizes the process
\[
f(X^*_x(\tau)[a,b])e^{-\la t} + \int_0^t e^{-\la s} (\la -A_{[a,b]\star}) f(X^*_x(s)[a,b]) \, ds,
\]
which is known (see e. g. \cite{EK})
to be a martingale for any $f\in C^1[a,b]$ (under the conclusions of Theorem  \ref{theorderonecapder}).
Again by the optional sampling theorem it follows that if $f$ solves \eqref{eqbaslinfracest}, then
\begin{equation}
\label{eqsolvetwosidlingen1}
f(x)= \E [f(X^*_x(\tau)[a,b])e^{-\la \tau}]
=f(a) \E [e^{-\la \tau} \1_{X_x(\tau)\le a}] +f(b) \E [e^{-\la \tau} \1_{X_x(\tau)\ge b}].
\end{equation}

Formula \eqref{eqsolvetwosidlingen1} again enures uniqueness of classical solution to \eqref{eqbaslinfracest}
and yields its integral representation in case the expectations involved in the r.h.s. of \eqref{eqsolvetwosidlingen1}
are sufficiently regular functions of $x$.

As an example, let us consider problem \eqref{eqtwosidedfracint} on the interval $[a,b]=[-1,1]$.
The corresponding process $X_x(t)$ of Theorem \ref{theorderwelpos1} is a symmetric
L\'evy motion on $\R$ with index $\be \in (0,1)$. For this process
all ingredients of formula \eqref{eqsolvetwosidgen3} are known, see e. g. \cite{BlGeRa}:
\begin{equation}
\label{eqBlumGet1}
\P (X_x(\tau) \ge 1)=2^{1-\be} \frac{\Ga(\be)}{(\Ga (\be/2))^2} \int_{-1}^x (1-u^2)^{-1+\be/2} \, du,
\end{equation}
and $H(x,dy)$ has the density
\begin{equation}
\label{eqBlumGet2}
H(x,y)=2^{-\be}\pi^{-1/2} \frac{\Ga(1/2)}{(\Ga (\be/2))^2} \int_0^z (u+1) ^{-1/2} u^{\be/2 -1}|x-y|^{\be -1} \, du,
\end{equation}
where
\[
z=(1-x^2)(1-y^2)/(x-y)^2.
\]

Thus \eqref{eqsolvetwosidgen3} yields a solution to problem \eqref{eqtwosidedfracint} in closed form.
Some formulas for exit probabilities are also available for nonsymmetric L\'evy motions, see
\cite{Rog72} and \cite{KypPaWa}, thus yielding explicit solutions to a slightly more general
(compared to \eqref{eqtwosidedfracint}) problem
\begin{equation}
\label{eqtwosidedfracintnonsym}
\al_1D^{\be}_{a+\star}f(x)+\al_2 D^{\be}_{b-\star}f(x)=-g(x), \quad f(a)=f_a, f(b)=f_b,
\end{equation}
with arbitrary positive constants $\al_1, \al_2$.

Similar results hold for the equations on a half-line involving the operators $A_{a\pm \star}$.

\subsection{The case $\be \in (0,1), d>1$}

Let us now consider operator \eqref{eqbaslinfracestlin2} assuming again for simplicity
that the kernel $\nu$ has a density, $\nu(x,y)$, with respect to Lebesgue measure.

The killed processes generated by \eqref{eqgeninterbeta01genmultdimRL} (which are well studied
for L\'evy processes, see e. g. \cite{BoBuCh03}) are generally easier for
analysis than stopped processes. Therefore
we concentrate on the analysis of operator  \eqref{eqgeninterbeta01genmultdimC},
which is more involved.
Let us consider only the case when $D$ is the semi-space $D_b$
or the band $D_{(a,b)}$, see \eqref{eqdefsemisp}, \eqref{eqdefband},
where $A_{D\star}$ is given by \eqref{eqCapsemisp} and \eqref{eqCapband}.

In what follows we have to use a rather ugly additional condition
\begin{equation}
\label{eqCapsemispcondug}
\Om (\ep,x)=\ep\int_{\ep}^{\infty} dy_1 \int_{\R^d} dy_2 \1_{|y|\le 1} \nu (x;y_1,y_2)\frac{y_2^2}{y_1^2}
\le C \om (\ep,x)
\end{equation}
with a constant $C$, where
\[
 \om (\ep,x)= \int_{\ep}^{\infty} dy_1 \int_{\R^d} dy_2 \nu (x;y_1,y_2).
 \]
Its reasonability relies on the fact that it holds for stable-like processes in $\R^{d+1}$ with
\begin{equation}
\label{eqdefstablelikeden}
\nu (x;y)=\frac{a(x)}{|y|^{d+1+\be}}, \quad \be \in (0,1),
\end{equation}
and hence for a variety of standard examples.

In fact, since
\[
\int_{\ep}^{\infty} dy_1 \int_{\R^d} dy_2 \1_{|y|\le 1}
=\int_{\ep}^1 r^d \, dr \int_0^{\arccos (r)} \sin^{d-1} \phi\, d\phi \, dn,
\]
where $dn$ is Lebesgue measure on the unit sphere in $\R^d$
(or just a coefficient $2$ in case $d=1$)
with the total area $|S^{d-1}|$,
 one has, assuming \eqref{eqdefstablelikeden}, that
 \[
 \Om (\ep,x)\le a(x) |S^{d-1}| \ep \int_{\ep}^1 r^{-1-\be} \, dr \int_0^{\arccos (r)} d\phi \frac{\sin \phi}{\cos ^2 \phi}
 \]
 \[
 =a(x) |S^{d-1}| \ep \int_{\ep}^1 r^{-1-\be} \, dr \int_r^1 \frac{dz}{z^2} \le \frac{a(x)}{1+\be} \ep^{-\be}
 \]
 and
 \[
 \om (\ep,x) \ge a(x) |S^{d-1}| \int_{\ep}^1 r^{-1-\be} \, dr \int_0^{\arccos (r)} \sin \phi \, d\phi
 \]
 \[
 =a(x) |S^{d-1}| \ep \int_{\ep}^1 r^{-1-\be} \, dr \int_r^1 dz
 = a(x) \left[ \frac{\ep^{-\be}-1}{\be}-\frac{1-\ep^{1-\be}}{1-\be}\right],
 \]
 implying \eqref{eqCapsemispcondug}.

Extending one-dimensional notations for function spaces used above we shall denote by $C_{\infty}[\bar D]$
 the Banach space of continuous
functions on $\bar D$  vanishing at infinity and by $C^1_{\infty}[\bar D]$ its subspace of functions
with first order partial derivatives belonging to $C_{\infty}[\bar D]$.

\begin{theorem}
\label{theorderonecapdermult}
Assume that $\nu(x;y)=\nu(x_1,x_2;y_1,y_2)$ is a continuous function, which is continuously differentiable
with respect to $x$ and has uniform bounds \eqref{eq1theorderonecapder}
and tightness property \eqref{eq2theorderonecapder}, where the integrals are over $\R^{d+1}$,
and additionally the bound on the second derivative with respect to $x_2$:
\begin{equation}
\label{eq1theorderonecapdersecder}
\quad \sup_x \int |y|\, \left|\frac{\pa^2}{\pa x_2^2}\nu (x_1,x_2;y_1,y_2)\right| \, dy <\infty.
\end{equation}
Moreover, assume that \eqref{eqCapsemispcondug} holds with a constant $C$ and that $\nu (x;y_1,y_2)$
is an even function of $y_2$.

(i) Then, for a semi-space $D=D_b$, the operator $A_{D\star}$ generates a Feller process $X^*_t(x)$ on $\bar D$
and a Feller semigroup on $C_{\infty}[\bar D]$ with invariant core $C^1_{\infty}[\bar D]$.

(ii) If additionally \eqref{eq1theorderonecapdersecder} holds also with
 the integral in $y_1$ taken over $(-\infty, \ep)$,
then also for a band $D=D_{(a,b)}$ the operator $A_{D\star}$
generates a Feller process $X^*_t(x)$ on $\bar D$
and a Feller semigroup on $C_{\infty}[\bar D]$ with invariant core $C^1_{\infty}[\bar D]$.

\end{theorem}

\begin{proof}
The proof follows the same lines as in Theorem \ref{theorderonecapder} above.
Let us deal only with domain $D_b$.
Approximating $\nu$ by $\nu_h(x;y)=\1_{y_2>h} \nu (x;y)$ we get semigroups $T_h$
on $C_{\infty}(\bar D)$, which are uniformly bounded and preserves twice continuous differentiability
with respect to the second variable $x_2$ and bounds to these derivatives. This holds because
differentiation of the equation $\dot f = A_{\D\star} f$ with respect to $x_2$
 does not feel the boundary so-to-say,
 that is we get
 \[
 \frac{d}{dt} \frac{\pa f}{\pa x_2}
 =   \int_{\R^d} dy_2 \int_{-\infty}^{b-x_1} dy_1 \nu (x;y)[\frac{\pa f}{\pa x_2}(x+y)-\frac{\pa f}{\pa x_2}(x)]
\]
\[
+\int_{\R^d} dy_2 \int_{b-x_1}^{\infty} dy_1 \nu (x;y)[\frac{\pa f}{\pa x_2}\left(b, x_2 +\frac{b-x_1}{y_1} y_2 \right)-\frac{\pa f}{\pa x_2}(x)]
\]
\[
+\int_{\R^d} dy_2 \int_{-\infty}^{b-x_1} dy_1 \frac{\pa \nu}{\pa x_2} (x;y)[f(x+y)-f(x)]
\]
\[
+\int_{\R^d} dy_2 \int_{b-x_1}^{\infty} dy_1 \frac{\pa \nu}{\pa x_2} (x;y)[f\left(b, x_2 +\frac{b-x_1}{y_1} y_2 \right)-f(x)],
\]
and similarly for the second derivative in $x_2$ (with $h$ and without it).
The problem arises when differentiating the
equation $\dot f = A_{\D\star} f$ with respect to $x_1$ yielding the equation
\[
 \frac{d}{dt}g
 = \int_{\R^d} dy_2 \int_{-\infty}^{b-x_1} dy_1 \nu (x;y)[g(x+y)-g(x)]
-g(x)\int_{\R^d} dy_2 \int_{b-x_1}^{\infty} dy_1 \nu (x;y)
\]
\begin{equation}
\label{eq2theorderonecapdermult}
-\int_{\R^d} dy_2 \int_{b-x_1}^{\infty} dy_1 \nu (x;y) \frac{y_2}{y_1}\frac{\pa f}{\pa x_2}\left(b, x_2 +\frac{b-x_1}{y_1} y_2 \right)
\end{equation}
(other terms cancel as in one-dimensional case)
for $g=\pa f/\pa x_1$. Similar equation hods for $\nu_h$ instead of $\nu$.
Because of assumed symmetry of $\nu$ this rewrites as
\begin{equation}
\label{eq3theorderonecapdermult}
 \frac{d}{dt}g(x)
= \int_{\R^d} dy_2 \int_{-\infty}^{b-x_1} dy_1 \nu (x;y)[g(x+y)-g(x)]
-g(x)\om (b-x_1,x) +\phi(x;f)
\end{equation}
with
\[
\om (\ep,x)= \int_{\R^d} dy_2 \int_{\ep}^{\infty} dy_1 \nu (x;y),
\]
\[
\phi(x;f)=-\int_{\R^d} dy_2 \int_{b-x_1}^{\infty} dy_1 \nu (x;y)
\frac{(b-x_1)y_2^2}{y_1^2}\frac{\pa ^2f}{\pa x_2^2}\left(b, x_2 +\theta \frac{b-x_1}{y_1} y_2 \right),
\]
where $\theta \in (0,1)$ so that, for $\pa^2 f/\pa x^2$ bounded by a constant $c$,
$\phi$ is bounded:
\[
|\phi (x;f)|\le c(b-x_1)\int_{\R^d} dy_2 \int_{b-x_1}^{\infty} dy_1 \nu (x;y)
\frac{y_2^2}{y_1^2}.
\]
Equation \eqref{eq3theorderonecapdermult} rewrites in the mild form as
\begin{equation}
\label{eq4theorderonecapdermult}
 g_t(x)=g_0(x)+\int_0^t e^{-\om (b-x_1,x)(t-s)} A g_s(x) \, ds +\int_0^t  e^{-\om (b-x_1,x)(t-s)} \phi (x;f) \, ds,
\end{equation}
where
\[
Ag(x)=\int_{\R^d} dy_2 \int_{-\infty}^{b-x_1} dy_1 \nu (x;y)[g(x+y)-g(x)].
\]
By \eqref{eqCapsemispcondug}, the last term in \eqref{eq4theorderonecapdermult} is uniformly bounded,
and hence equation  \eqref{eq4theorderonecapdermult} and its versions with $\nu_h$ instead of $\nu$
have uniformly bounded solutions for bounded $g_0$. Hence we can now complete the proof as in Theorem \ref{theorderonecapder}.
\end{proof}

One can now get a direct multi-dimensional version of Theorem \ref{theorderwelpos1}
for the boundary value problems
\begin{equation}
\label{eq5theorderonecapdermult}
A_{D\star} f= g, \quad f|_{\pa D}=\phi,
\end{equation}
with $g$ in $D$ and $\phi $ on $\pa D$ given, which represent the simplest multidimensional analogs
of linear equations with the Caputo derivatives.
Alternatively, one can also analyze such problems
via the reduction to killed processes (that is, to the analogs of RL derivatives),
 see Remark \ref{remRLClink}.

 We shall not go into detail of general domains $D$ here, but note that the problem
\begin{equation}
\label{eq6theorderonecapdermult}
A_{D\star} f= \la f+ g, \quad f|_{\pa D}=\phi,
\end{equation}
for $A$ generating a stable L\'evy process in $\R^d$ and $D$ the ball $D=\{y\in \R^d: |y|<r\}$
can be solved explicitly, using multidimensional extensions of formulas
\eqref{eqBlumGet1} and \eqref{eqBlumGet2} given also in \cite{BlGeRa}.

\subsection{The case $\be \in (1,2)$}

For the case $\be\in (0,1)$ above we constructed the interrupted process on its own
and then look at its stopping, which is quite natural. However, as we noted, the boundary-value
problems (at least in one-dimensional case) for corresponding operators can
be expressed in terms of the initial process stopped at the boundary. We shall follow this
approach here, as the study of interrupted process becomes rather subtle.

Let us reduce our attention to one-dimensional processes only generated by the operators
\[
Af(x)=\int [f(x+y-f(x)-yf'(x)]\nu (x,y) \, dy
\]
with the density $\nu$ satisfying
\begin{equation}
\label{eqpropdenlevy1}
\sup_x \int_{\R} (|y|\land |y|^2) \nu (x, y) \, dy <\infty.
\end{equation}
The question of whether such an operator generates a uniquely defined process
is non-trivial already in the case of this simple $A$, which can be looked at as the fully mixed-order
fractional derivative. To go ahead, we shall use additional
assumptions of regularity and monotonicity.
The following statement is a particular case of Theorem 4.1 of \cite{Ko11a}:

\begin{prop}
\label{propfrommonrus}
Assume that $\nu$ is twice continuously differentiable
with respect to the first variable satisfying
\begin{equation}
\label{eqpropdenlevy2}
\sup_x \int (|y|\land |y|^2) |\frac{\pa}{\pa x} \nu (x,y)| \, dy <\infty, \quad
\sup_x \int (|y|\land |y|^2) |\frac{\pa^2}{\pa x^2} \nu (x, y)| \, dy <\infty,
\end{equation}
and that the functions
\begin{equation}
\label{eqcondmonotonlevymeasure}
\int_a^{\infty} \nu (x,y), \, dy \quad \int_{-\infty}^{-a} \nu (x,y) \, dy
\end{equation}
are non-decreasing and non-increasing respectively for any $a>0$.
Then the operator $A$ generates a Feller process $X_t(x)$ on $\R$
and a Feller semigroup with the space $C^2_{\infty}(\R)$ being an invariant core. The process $X_t(x)$
is stochastically monotone (but we will not use this latter fact).
\end{prop}

Next let $-\infty < a <x < b <\infty$ and let
\begin{equation}
\label{eqgeninterboundintbelarge}
\tilde A_{[a,b]}f(x)= \int_{a-x}^{b-x}  [f(a\vee [(x+y) \wedge b])-f(x)-yf'(x)] \nu (x, y) \, dy
\end{equation}
be the operator representing the corresponding process $X^*_x(t)[a,b]$ interrupted on an attempt to cross
the boundary of $[a,b]$ (the processes on $(-\infty, b]$ or $[a,\infty)$ with a one-sided boundary
are considered analogously and will not be looked at) and
\begin{equation}
\label{eqgeninterboundintbelargeCap}
A_{[a,b]\star}f(x)= \tilde A_{[a,b]}f(x)-f'(b) \int_{b-x}^{\infty} [(b-x)-z]\nu(x, z)\, dz
-f'(a)\int_{-\infty}^{a-x} [(a-x)-z]\nu (x,z) \, dz
\end{equation}
the corresponding analog of Caputo's derivative (see \eqref{eqlevyneginterabreg}).

As above, Proposition \ref{propfrommonrus}  allows us to apply the standard tools of stochastic calculus.
Namely, let $f \in C^2[a,b]$ such that $f'(a)=f'(b)=0$. Then we can continue it to all $\R$ by setting
$f(x)=f(b)$ for $x>b$ and $f(x)=f(a)$ for $x<a$ and it will become a bounded continuously differentiable function on $\R$.
Denoting as above by $\tau $ the exit time from $(a,b)$, that is $\tau=\inf\{t: X_t(x) \notin (a,b)\}$ and
applying to $f$ Dynkin's martingale we get again \eqref{eqsolvetwosidgen2},
or else \eqref{eqsolvetwosidgen3}, using the kernel $H(x,dy)$ defined by \eqref{eqdefoccuptllexit}
and assuming $f$ solves problem
\eqref{eqbaslinfracestrep2} with $\la =0$.

\begin{remark}
Actually we can use Dynkin's martingale only
for twice continuously differentiable functions, and our (extended) $f$ may have discontinuities of the second derivatives
 on the boundary, but this can be settled via approximation, as the final expression \eqref{eqsolvetwosidgen3}
 does not involve the second derivative of $f$ on the boundary.
\end{remark}

Therefore we get the following version of Theorem \ref{theorderwelpos1} for the present case $\be \in (1,2)$:
 \begin{theorem}
\label{theorderwelpos1belarge}
(i) Under the assumptions of Proposition  \ref{propfrommonrus} problem
\eqref{eqbaslinfracestrep2} with $\la =0$
can have at most one classical solution. If the probabilities $\P (X_x(\tau)\le a)$, $\P (X_x(\tau)\ge b)$ are
functions from $C^2[a,b]$ and the measure $H(x,dy)$ is continuously weakly differentiable in $x$ (so that the integral
on the r.h.s. of \eqref{eqsolvetwosidgen3} belongs to $C^2[a,b]$ for any $g\in C[a,b]$),
formula \eqref{eqsolvetwosidgen3} supplies the unique classical solution to \eqref{eqbaslinfracestrep2} with $\la =0$.

(ii) Generally under the assumptions of Proposition  \ref{propfrommonrus},
formula \eqref{eqsolvetwosidgen3} supplies a unique generalized solution to \eqref{eqbaslinfracestrep2} with $\la =0$.
\end{theorem}

As was noted solutions to  \eqref{eqbaslinfracestrep2} solve also \eqref{eqbaslinfracestrep1} under the
additional assumption that the first derivative of the solution vanishes on the boundary, which yields
a rather tamed (and expected) non-uniqueness for \eqref{eqbaslinfracestrep1}.

\section{Appendix}

For completeness we deduce here the expressions \eqref{eqRLder0}, \eqref{eqCder0}
and \eqref{eqRLder1}, \eqref{eqCder1} from the original definitions.

For $\be \in (0,1)$ and $x>a$ integration by parts yields
\[
I^{1-\be}_{a+}f(x)=\frac{1}{\Ga (1-\be)}\int_a^x (x-t)^{-\be}f(t) \, dt
=-\frac{1}{\Ga (1-\be)} \int_a^x \frac{d}{dt} \left[\frac{(x-t)^{1-\be}}{1-\be}\right]f(t) \, dt
\]
\[
=\frac{1}{\Ga (1-\be)} \left[\frac{(x-a)^{1-\be}}{1-\be}\right]f(a)
+\frac{1}{\Ga (1-\be)} \int_a^x \left[\frac{(x-t)^{1-\be}}{1-\be}\right]f'(t) \, dt,
\]
so that
\begin{equation}
\label{eqcalRLder}
D^{\be}_{a+}f(x)=\frac{d}{dx} I_{a+}^{1-\be}f(x)
=\frac{f(a)}{\Ga (1-\be)(x-a)^{\be}}
+\frac{1}{\Ga (1-\be)} \int_a^x (x-t)^{-\be}f'(t) \, dt.
\end{equation}
Another integration by parts using
\[
f'(t) =\frac{d}{dt} (f(t)-f(x)),
\]
yields
\[
D^{\be}_{a+}f(x)
=\frac{f(a)}{\Ga (1-\be)(x-a)^{\be}}-\frac{f(a)-f(x)}{\Ga (1-\be)(x-a)^{\be}}
-\frac{\be}{\Ga (1-\be)} \int_a^x \frac{f(t)-f(x)}{(x-t)^{1+\be}} \, dt,
\]
which equals to the r.h.s. of \eqref{eqRLder0}.
On the other hand,
\[
D^{\be}_{a+\star}f(x)=I_{a+}^{1-\be}f'(x)
=\frac{1}{\Ga (1-\be)}\int_a^x (x-t)^{-\be}f'(t) \, dt,
\]
which differs from \eqref{eqcalRLder} by $f(a)(x-a)^{-\be} / \Ga (1-\be)$ yielding \eqref{eqCder0},
\eqref{eqRLCapder}.

For $\be \in (1,2)$ and $x>a$ one has
\begin{equation}
\label{eqcalRLder1}
D^{\be}_{a+\star}f(x)
=\frac{1}{\Ga (2-\be)}\int_a^x (x-t)^{1-\be }f''(t) \, dt,
\end{equation}
which rewrites as
\[
=\frac{1}{\Ga (2-\be)}(x-a)^{1-\be }(f'(x)-f'(a)) +\frac{1-\be}{\Ga (2-\be)}\int_a^x (x-t)^{-\be }(f'(t)-f'(x))\, dt.
\]
Another integration by parts using
\[
f'(t)-f'(x)=\frac{d}{dt}(f(t)-f(x)-(t-x) f'(x))
\]
yields
\[
D^{\be}_{a+\star}f(x)
= \frac{1}{\Ga (2-\be)}(x-a)^{1-\be }(f'(x)-f'(a))
\]
\[
-\frac{1}{\Ga (1-\be)}(x-a)^{-\be }(f(a)-f(x)-(a-x)f'(x))
+\frac{1}{\Ga (-\be)}\int_a^x \frac{f(t)-f(x)-(t-x) f'(x)}{(x-t)^{1+\be} } \, dt,
\]
which equals the r.h.s. of \eqref{eqCder1}.

On the other hand, again
for $\be \in (1,2)$ and $x>a$,
\[
I^{2-\be}_{a+}f(x)=\frac{1}{\Ga (2-\be)}\int_a^x (x-t)^{1-\be }f(t) \, dt,
\]
which rewrites by integration by parts as
\[
I^{2-\be}_{a+}f(x)=\frac{f(a)}{\Ga (2-\be)}\frac{(x-a)^{2-\be }}{2-\be}
+\frac{1}{\Ga (2-\be)}\int_a^x \frac{(x-t)^{2-\be }}{2-\be} f'(t) \, dt,
\]
and by yet another integration by parts as
\[
I^{2-\be}_{a+}f(x)=\frac{f(a)}{\Ga (2-\be)}\frac{(x-a)^{2-\be }}{2-\be}
+\frac{f'(a)(x-a)^{3-\be }}{\Ga (2-\be)(2-\be)(3-\be)}
+\int_a^x \frac{(x-t)^{3-\be }f''(t)}{\Ga (2-\be)(2-\be)(3-\be)} \, dt.
\]
Consequently,
\[
D^{\be}_{a+}f(x)=\frac{d^2}{dx^2}I^{2-\be}_{a+}f(x)
=\frac{1-\be}{\Ga (2-\be)}(x-a)^{-\be }f(a)
+\frac{f'(a)(x-a)^{1-\be }}{\Ga (2-\be)}
+\int_a^x \frac{(x-t)^{1-\be }f''(t)}{\Ga (2-\be)} \, dt.
\]
Comparing this with \eqref{eqcalRLder1} yields
\eqref{eqRLder1} and \eqref{eqRLCder2}.

Similarly, for $\be \in (1,2)$ and $x<a$ one has by definition \eqref{eqdefCderleft2} that
\[
D^{\be}_{a-\star}f(x)
=\frac{1}{\Ga (2-\be)}\int_x^a (t-x)^{1-\be }f''(t) \, dt
\]
\[
=-\frac{1}{\Ga (2-\be)}(a-x)^{1-\be }(f'(x)-f'(a)) -\frac{1-\be}{\Ga (2-\be)}\int_x^a (t-x)^{-\be }(f'(t)-f'(x))\, dt.
\]
By integration by parts this rewrites as
\[
D^{\be}_{a-\star}f(x)
=- \frac{1}{\Ga (2-\be)}(x-a)^{1-\be }(f'(x)-f'(a))
\]
\[
-\frac{1}{\Ga (1-\be)}(a-x)^{-\be }(f(a)-f(x)-(a-x)f'(x))
+\frac{1}{\Ga (-\be)}\int_x^a \frac{f(t)-f(x)-(t-x) f'(x)}{(t-x)^{1+\be} } \, dt,
\]
which equals the r.h.s. of \eqref{eqCder1left}.
Similarly, \eqref{eqRLder1left} is obtained.

\vspace{2mm}

{\bf Acknowledgements.} I am grateful to J.  Lorinczi, G. Pagnini and E. Scalas for inviting me 
to deliver a mini-course on probabilistic treatment of fractional differential equations in 
a school organized by the Basque Center for Applied Mathematics (BCAM) in Bilbao
(November 2014), as well as to D. Chamorro and S. Menozzi for inviting me to give a talk on this subject 
 on a workshop 'Analyse et Probabilit\'es' (Universit\'e d'Evry Val d'Essonne, October 2014).
 These stimulating events gave me the nice opportunity to think systematically about the topic 
 of the present paper.

\end{document}